\documentclass[11pt, reqno]{amsart} \textwidth=14.9cm \oddsidemargin=1cm \evensidemargin=1cm
\usepackage{amsmath,amssymb,amscd,mathrsfs,stmaryrd,wasysym}
\usepackage{amsmath}
\usepackage{amsxtra}
\usepackage{amscd}
\usepackage{amsthm}
\usepackage{amsfonts}
\usepackage{amssymb}
\usepackage{eucal}
\usepackage{color}
\usepackage{graphicx}%
\usepackage{bm}
\newtheorem{theorem}{Theorem}[section]
\newtheorem{lemma}[theorem]{Lemma}
\newtheorem{proposition}[theorem]{Proposition}

\theoremstyle{definition}

\newtheorem{remark}[theorem]{Remark}

\definecolor{A}{rgb}{.75,1,.75}

\newcommand{\gr}{\operatorname{gr}}
\numberwithin{equation}{section}

\begin{document}

\title[An isomorphism theorem]{An isomorphism theorem for degenerate cyclotomic Yokonuma-Hecke algebras and applications}
\author[Weideng Cui]{Weideng Cui}
\address{School of Mathematics, Shandong University, Jinan, Shandong 250100, P.R. China.}
\email{cwdeng@amss.ac.cn}

\begin{abstract}
Inspired by the work [PA], we establish an explicit algebra isomorphism between the degenerate cyclotomic Yokonuma-Hecke algebra $Y_{r,n}^{d}(q)$ and a direct sum of matrix algebras over tensor products of degenerate cyclotomic Hecke algebras of type $A$. We then develop several applications of this result, including a new proof of the modular representation theory of $Y_{r,n}^{d}(q)$, a semisimplicity criterion for it and cellularity of it. Moreover, we prove that $Y_{r,n}^{d}(q)$ is a symmetric algebra and determine the associated Schur elements by using the isomorphism theorem for it.
\end{abstract}



\maketitle
\medskip
\section{Introduction}
\subsection{}
Yokonuma-Hecke algebras were introduced by Yokonuma \cite{Yo} as a centralizer algebra associated to the permutation representation of a finite Chevalley group $G$ with respect to a maximal unipotent subgroup of $G$. The Yokonuma-Hecke algebra $Y_{r,n}(q)$ (of type $A$) is a quotient of the group algebra of the modular framed braid group $(\mathbb{Z}/r\mathbb{Z})\wr B_{n},$ where $B_{n}$ is the braid group on $n$ strands (of type $A$). By the presentation given by Juyumaya and Kannan \cite{Ju1, Ju2, JuK}, the Yokonuma-Hecke algebra $Y_{r,n}(q)$ can also be regraded as a deformation of the group algebra of the complex reflection group $G(r,1,n),$ which is isomorphic to the wreath product $(\mathbb{Z}/r\mathbb{Z})\wr \mathfrak{S}_{n}$. It is well-known that there exists another deformation of the group algebra of $G(r,1,n),$ namely the Ariki-Koike algebra \cite{AK}. The Yokonuma-Hecke algebra $Y_{r,n}(q)$ is quite different from the Ariki-Koike algebra.  For example, the Iwahori-Hecke algebra of type $A$ is canonically a subalgebra of the Ariki-Koike algebra, whereas it is an obvious quotient of $Y_{r,n}(q),$ but not an obvious subalgebra of it.

Recently, by generalizing the approach of Okounkov-Vershik \cite{OV} on the representation theory of the symmetric group $\mathfrak{S}_n$, Chlouveraki and Poulain d'Andecy \cite{ChPA1} introduced the notion of affine Yokonuma-Hecke algebra $\widehat{Y}_{r,n}(q)$ and gave explicit formulas for all irreducible representations of $Y_{r,n}(q)$ over $\mathbb{C}(q)$, and obtained a semisimplicity criterion for it. In their subsequent paper [ChPA2], they studied the representation theory of the affine Yokonuma-Hecke algebra $\widehat{Y}_{r,n}(q)$ and the cyclotomic Yokonuma-Hecke algebra $Y_{r,n}^{d}(q)$. In particular, they gave the classification of irreducible representations of $Y_{r,n}^{d}(q)$ in the generic semisimple case. In [CW], we gave the classification of the simple $\widehat{Y}_{r,n}(q)$-modules as well as the classification of the simple modules of the cyclotomic Yokonuma-Hecke algebras over an algebraically closed field $\mathbb{K}$ of characteristic $p$ such that $p$ does not divide $r.$ Rostam [Ro] proved that the cyclotomic Yokonuma-Hecke algebra is a particular case of cyclotomic quiver Hecke algebras. In the past several years, the study of affine and cyclotomic Yokonuma-Hecke algebras has made substantial progress; see [ChPA1, ChPA2, ChS, C1, C2, CW, ER, JaPA, Lu, PA, Ro].

\subsection{}
Schur elements play a powerful role in the representation theory of symmetric algebras (see, e.g., [GP, Chapter 7]). Malle and Mathas [MM] showed that the cyclotomic Hecke algebra (Ariki-Koike algebra) is a symmetric algebra. Formulae for the Schur elements of the cyclotomic Hecke algebra have been obtained independently, first by Geck, Iancu and Malle [GIM], and later by Mathas [M]; see also [ChJa].

In the case of the degenerate cyclotomic Hecke algebra, Brundan and Kleshchev [BK1] proved that it is a symmetric algebra for all parameters. Later on, Zhao [Z1-2] gave an explicit combinatorial formula for the Schur element of the degenerate cyclotomic Hecke algebra.

Thus, we can use the Schur elements to determine when Specht  modules are projective irreducible and whether the cyclotomic and degenerate cyclotomic Hecke algebra are semi-simple.

\subsection{}
Largely inspired by the work [PA], we establish an explicit algebra isomorphism between the degenerate cyclotomic Yokonuma-Hecke algebra $Y_{r,n}^{d}(q)$ and a direct sum of matrix algebras over tensor products of degenerate cyclotomic Hecke algebras of type $A$. We then develop several applications of this result, including a new proof of the modular representation theory of $Y_{r,n}^{d}(q)$, a semisimplicity criterion for it and cellularity of it. Moreover, we prove that $Y_{r,n}^{d}(q)$ is a symmetric algebra and determine the associated Schur elements by using the isomorphism theorem for it.

This paper is organized as follows. In Section 2, we recall some necessary results and prove that the degenerate cyclotomic Yokonuma-Hecke algebra is a symmetric algebra following [BK1, Appendix A]. In Section 3, we establish an explicit algebra isomorphism between the degenerate affine (resp. cyclotomic) Yokonuma-Hecke algebra $\widehat{Y}_{r,n}(q)$ (resp. $Y_{r,n}^{d}(q)$) and a direct sum of matrix algebras over tensor products of degenerate affine (resp. cyclotomic) Hecke algebras of type $A$ following the approach of Poulain d'Andecy. In Section 4, we develop several applications of the algebra isomorphism.

\section{Degenerate cyclotomic Yokonuma-Hecke algebras are symmetric}

In this section, we first recall the definition of the degenerate cyclotomic Yokonuma-Hecke algebra $Y_{r, n}^{d}$, and then prove that it is a symmetric algebra following the approach in [BK1, Appendix A].

\subsection{Degenerate cyclotomic Hecke algebras}
Let $n, d\in \mathbb{Z}_{\geq 1}$ and let $\mathbb{K}$ be an algebraically closed field of characteristic $p\geq 0$ which contains some elements $v_{1},\ldots, v_{d}.$ The degenerate affine Hecke algebra $\widehat{H}_{n}$ is the associative $\mathbb{K}$-algebra generated by the elements $\bar{x}_1,\ldots, \bar{x}_{n}$ and $\bar{s}_1,\ldots, \bar{s}_{n-1}$ subject to the following relations:
\begin{align}
\label{ECoxeter1}
\bar{s}_r^2&=1;\\
\label{ECoxeter2}
\bar{s}_r\bar{s}_{r+1}\bar{s}_r&=\bar{s}_{r+1}\bar{s}_r\bar{s}_{r+1},
\qquad
\bar{s}_r\bar{s}_t=\bar{s}_t\bar{s}_r\hspace{3mm} \text{if $|r-t|>1$};\\
\label{EPoly}
\bar{x}_r\bar{x}_t&=\bar{x}_t\bar{x}_r;
\\\label{EDAHA}
\bar{s}_r \bar{x}_{r+1} &= \bar{x}_r \bar{s}_r + 1,\hspace{10.5mm} \bar{s}_r \bar{x}_t = \bar{x}_t \bar{s}_r \hspace{1.5mm}
\text{ if $t \neq r,r+1$}.
\end{align}
By [K, Theorem 3.2.2], $\widehat{H}_{n}$ is a free $\mathbb{K}$-module with basis
\begin{equation}\label{Ahecke-basis}
\big\{\bar{x}_{1}^{\alpha_1}\cdots \bar{x}_{n}^{\alpha_{n}}\bar{w}\:|\:\alpha_1,\ldots,\alpha_{n}\geq 0, w\in \mathfrak{S}_{n}\big\}.
\end{equation}

The degenerate cyclotomic Hecke algebra $H_{n}^{d}$ is defined to be the quotient:\[H_{n}^{d}=\widehat{H}_{n}/\langle (\bar{x}_1-v_1)\cdots (\bar{x}_{1}-v_{d})\rangle.\]
By [K, Theorem 7.5.6], $H_{n}^{d}$ is a free $\mathbb{K}$-module with basis
\begin{equation}\label{Checke-basis}
\big\{\bar{x}_{1}^{i_1}\cdots \bar{x}_{n}^{i_{n}}\bar{w}\:|\:0\leq i_1,\ldots,i_{n}<d, w\in \mathfrak{S}_{n}\big\}.
\end{equation}

Let $\tau_n: H_{n}^{d}\rightarrow \mathbb{K}$ be the $\mathbb{K}$-linear map defined by
\begin{align}\label{symme-form}
\tau_n(\bar{x}_{1}^{i_1}\cdots \bar{x}_{n}^{i_{n}}\bar{w})=
\begin{cases}
1 & \text{if } i_{1}=\cdots=i_{n}=d-1\text{ and }w=1,
\\
0 & \text{otherwise. }
\end{cases}
\end{align}
By [BK1, Appendix], $\tau_n$ is a non-degenerate trace on $H_{n}^{d}$ for all parameters $v_{1},\ldots, v_{d}$ in $\mathbb{K}$, that is, $H_{n}^{d}$ is a symmetric algebra for all parameters.

\subsection{Degenerate cyclotomic Yokonuma-Hecke algebras}
In the rest of this paper, we shall assume that $p$ does not divide $r$ for some $r\in \mathbb{Z}_{\geq 1}$. The degenerate affine Yokonuma-Hecke algebra, denoted by $\widehat{Y}_{r,n}$, is an associative $\mathbb{K}$-algebra generated by the elements $t_{1},\ldots,t_{n},f_{1},\ldots,f_{n-1}, x_1,\ldots, x_{n}$ in which the generators $t_{1},\ldots,t_{n},f_{1},$ $\ldots,f_{n-1}$ satisfy the following relations:
\begin{align}
f_if_j&=f_jf_i \qquad \qquad\qquad\quad\hspace{0.3cm}\mbox{for all $i,j=1,\ldots,n-1$ such that $\vert i-j\vert \geq 2$;}\label{drel-def-Y1}\\[0.1em]
f_if_{i+1}f_i&=f_{i+1}f_if_{i+1} \qquad \quad\qquad\hspace{0.05cm}\mbox{for all $i=1,\ldots,n-2$;}\label{drel-def-Y2}\\[0.1em]
t_it_j&=t_jt_i \qquad\qquad\qquad\qquad  \mbox{for all $i,j=1,\ldots,n$;}\label{drel-def-Y3}\\[0.1em]
f_it_j&=t_{s_i(j)}f_i \quad \quad\qquad\qquad\hspace{0.25cm}\mbox{for all $i=1,\ldots,n-1$ and $j=1,\ldots,n$;}\label{drel-def-Y4}\\[0.1em]
t_i^r&=1 \quad \qquad\qquad\qquad\qquad\mbox{for all $i=1,\ldots,n$;}\label{drel-def-Y5}\\[0.2em]
f_{i}^{2}&=1 \qquad \hspace{2.71cm}\mbox{for all $i=1,\ldots,n-1$,}\label{drel-def-Y6}
\end{align}
together with the following relations concerning the generators $x_1,\ldots, x_{n}$:
\begin{align}
x_{i}x_{j}&=x_{j}x_{i};\label{drel-def-Y7}\\[0.1em]
f_{i}x_{i+1}&=x_{i}f_{i}+e_{i};\label{drel-def-Y8}\\[0.1em]
f_{i}x_{j}&=x_{j}f_{i} \qquad \qquad \quad\mbox{for all $j\neq i, i+1$;}\label{drel-def-Y9}\\[0.1em]
t_{j}x_{i}&=x_{i}t_{j} \hspace{0.07cm}\qquad \qquad\quad\mbox{for all $i, j=1,\ldots,n$,}\label{drel-def-Y10}
\end{align}
where $s_{i}$ is the transposition $(i,i+1)$, and for each $1\leq i\leq n-1$,
$$e_{i} :=\frac{1}{r}\sum\limits_{s=0}^{r-1}t_{i}^{s}t_{i+1}^{-s}.$$

\begin{remark}\label{rem-YH}
The degenerate affine Yokonuma-Hecke algebra $\widehat{Y}_{r,n}$ is in fact a special case of the wreath Hecke algebra $\mathcal{H}_{n}(G)$ defined in [WW, Definition 2.4] when $G=C_{r}$ is the cyclic group of order $r;$ see also [RS].
\end{remark}
\noindent By [WW, Theorem 2.8], $\widehat{Y}_{r,n}$ has a $\mathbb{K}$-basis
\begin{equation}\label{PBW-basis}
\big\{t_{1}^{\beta_1}\cdots t_{n}^{\beta_{n}}x_{1}^{\alpha_1}\cdots x_{n}^{\alpha_{n}}f_{w}\:|\:0\leq \beta_1,\ldots,\beta_{n}\leq r-1, \alpha_1,\ldots,\alpha_{n}\geq 0, w\in \mathfrak{S}_{n}\big\}.\end{equation}

We define the degenerate cyclotomic Yokonuma-Hecke algebra $Y_{r, n}^{d}$ to be the quotient:
\[Y_{r, n}^{d}=\widehat{Y}_{r,n}/\langle (x_1-v_1)\cdots (x_{1}-v_{d})\rangle.\]
By [WW, Proposition 5.5] (see also [Ro, Proposition 5.2]), $Y_{r, n}^{d}$ has a $\mathbb{K}$-basis
\begin{equation}\label{CYPBW-basis}
\big\{t_{1}^{j_1}\cdots t_{n}^{j_{n}}x_{1}^{i_1}\cdots x_{n}^{i_{n}}f_{w}\:|\:0\leq i_1,\ldots,i_{n}\leq d-1,0\leq j_1,\ldots,j_{n}\leq r-1, w\in \mathfrak{S}_{n}\big\}.\end{equation}

\subsection{Symmetric algebras}
Let $\mathbb{K}_{d}[x_1,\ldots,x_{n}]$ be the level $d$ truncated polynomial algebra, that is, the quotient of the polynomial algebra $\mathbb{K}[x_1,\ldots,x_{n}]$ by the two-sided ideal generated by $x_{1}^{d},\ldots,x_{n}^{d}.$ Define a filtration $F_{0}Y_{r, n}^{d}\subseteq F_{1}Y_{r, n}^{d}\subseteq \cdots$ by declaring that $F_{k}Y_{r, n}^{d}$ is the span of all $t_{1}^{j_1}\cdots t_{n}^{j_{n}}x_{1}^{i_1}\cdots x_{n}^{i_{n}}f_{w}$ for $0\leq j_1,\ldots,j_{n}\leq r-1$, $i_1,\ldots,i_{n}\geq 0$ and $w\in \mathfrak{S}_{n}$ with $i_1+\cdots+i_{n}\leq k.$

Let $W_{r,n}=(\mathbb{Z}/r\mathbb{Z})\wr \mathfrak{S}_{n}$ and $\mathbb{K}W_{r,n}$ be its group algebra, which has a set of generators $t_1,\ldots,t_{n},s_{1},\ldots,s_{n-1}.$ Consider the twisted tensor algebra $\mathbb{K}_{d}[x_1,\ldots,x_{n}]\,{\scriptstyle{\rtimes\!\!\!\!\!\bigcirc}}\, \mathbb{K}W_{r,n}$ and define a grading on it by declaring that each $x_{i}$ is of degree 1 and each $y\in W_{r,n}$ is of degree 0. Then we can easily get the following lemma.
\begin{lemma}\label{graded-algebra}
The map $\zeta_{r,d}:\mathbb{K}_{d}[x_1,\ldots,x_{n}]\,{\scriptstyle{\rtimes\!\!\!\!\!\bigcirc}}\, \mathbb{K}W_{r,n}\rightarrow \gr Y_{r, n}^{d},$ which is given by $x_i\mapsto \gr_1 x_i$ for each $1\leq i\leq n$, $t_i\mapsto \gr_0 t_i$ for each $1\leq i\leq n$ and $s_j\mapsto \gr_0 f_j$ for each $1\leq j\leq n-1,$ is an isomorphism of graded algebras.
\end{lemma}

Recall that a finite dimensional algebra $A$ is symmetric if it possesses a symmetrizing form, i.e. a linear form $\tau: A\rightarrow \mathbb{K}$ such that $\tau(ab)=\tau(ba)$ for all $a,b\in A$ and whose kernel contains no non-zero left or right ideal of $A.$ The following lemma can be regarded as a generalization of [BK1, Lemma A.1].
\begin{lemma}\label{twisted-algebra}
Let $\rho: \mathbb{K}_{d}[x_1,\ldots,x_{n}]\,{\scriptstyle{\rtimes\!\!\!\!\!\bigcirc}}\, \mathbb{K}W_{r,n}\rightarrow \mathbb{K}$ be the linear map sending the monomial $t_{1}^{s_1}\cdots t_{n}^{s_{n}}x_{1}^{r_1}\cdots x_{n}^{r_{n}}w$ to $1$ if $r_1=\cdots=r_{n}=d-1,$ $s_1\equiv\cdots\equiv s_{n}\equiv 0$ $(\bmod$ $r)$ and $w=1$, and to $0$ otherwise. Then $\rho$ is a symmetrizing form, hence $\mathbb{K}_{d}[x_1,\ldots,x_{n}]\,{\scriptstyle{\rtimes\!\!\!\!\!\bigcirc}}\, \mathbb{K}W_{r,n}$ is a symmetric algebra.
\end{lemma}

Now let $t=(d-1)^{n}$. Recall from Lemma \ref{graded-algebra} that $Y_{r, n}^{d}$ is a filtered algebra with a filtration $$\mathbb{K}W_{r,n}=F_{0}Y_{r, n}^{d}\subseteq F_{1}Y_{r, n}^{d}\subseteq \cdots\subseteq F_{t}Y_{r, n}^{d}=Y_{r, n}^{d},$$ and the associated graded algebra $\gr Y_{r, n}^{d}$ is identified with the twisted tensor product $\mathbb{K}_{d}[x_1,\ldots,x_{n}]\,{\scriptstyle{\rtimes\!\!\!\!\!\bigcirc}}\, \mathbb{K}W_{r,n}.$ For any $0\leq s\leq t$, let $\gr_s: F_{s}Y_{r, n}^{d}\rightarrow \mathbb{K}_{d}[x_1,\ldots,x_{n}]\,{\scriptstyle{\rtimes\!\!\!\!\!\bigcirc}}\, \mathbb{K}W_{r,n}$ be the map sending an element to its degree $s$ graded component.
\begin{theorem}\label{symmetric-algebra}
Let $\hat{\rho}: Y_{r, n}^{d}\rightarrow \mathbb{K}$ be the linear map sending $t_{1}^{s_1}\cdots t_{n}^{s_{n}}x_{1}^{r_1}\cdots x_{n}^{r_{n}}f_{w}$ to $1$ if $r_1=\cdots=r_{n}=d-1,$ $s_1\equiv\cdots\equiv s_{n}\equiv 0$ $(\bmod$ $r)$ and $w=1$, and to $0$ otherwise. Then $\hat{\rho}$ is a symmetrizing form, hence $Y_{r, n}^{d}$ is a symmetric algebra.
\end{theorem}
\begin{proof}
It follows from the key observation that $\hat{\rho}=\rho\circ \gr_t.$
\end{proof}

\section{An isomorphism theorem for degenerate cyclotomic Yokonuma-Hecke algebras}

Largely inspired by [PA, Section 3], we establish an explicit algebra isomorphism between the degenerate affine (resp. cyclotomic) Yokonuma-Hecke algebra $\widehat{Y}_{r,n}$ (resp. $Y_{r,n}^{d}$) and a direct sum of matrix algebras over tensor products of degenerate affine (resp. cyclotomic) Hecke algebras of type $A$ in this section.

\subsection{Preliminaries}
We first recall some constructions presented in [PA, Section 2]. Let $\mathcal{T}_{r, n}$ be the commutative subalgebra of $\widehat{Y}_{r,n}$ generated by $t_1,\ldots,t_{n}$ which is isomorphic to the group algebra of $(\mathbb{Z}/r\mathbb{Z})^{n}$ and let $\{\zeta_1,\ldots,\zeta_{r}\}$ be the set of all $r$-th roots of unity. A character $\chi$ of $\mathcal{T}_{r, n}$ over $\mathbb{K}$ is characterized by the choice of $\chi(t_{j})\in \{\zeta_1,\ldots,\zeta_{r}\}$ for each $j=1,\ldots,n.$ We denote by $\text{Irr}(\mathcal{T}_{r, n})$ the set of characters of $\mathcal{T}_{r, n}$ over $\mathbb{K}$, which is endowed with an action of $\mathfrak{S}_{n}$ by $w(\chi)(t_i)=\chi(t_{w^{-1}(i)}).$

For each $\chi\in \text{Irr}(\mathcal{T}_{r, n})$, let $E_{\chi}$ be the primitive idempotent of $\mathcal{T}_{r, n}$ associated to $\chi,$ which can be explicitly written in terms of the generators as follows:
\begin{equation}\label{idempotents}
E_{\chi}=\prod_{1\leq i\leq n}\bigg(\frac{1}{r}\sum_{0\leq s\leq r-1}\chi(t_{i})^{s}t_{i}^{-s}\bigg).\end{equation}
For $\alpha=(\alpha_1,\ldots,\alpha_{n})\in \mathbb{Z}_{\geq 0}^{n}$, we set $\bar{x}^{\alpha}=\bar{x}_{1}^{\alpha_{1}}\cdots \bar{x}_{n}^{\alpha_{n}}$ and $x^{\alpha}=x_{1}^{\alpha_{1}}\cdots x_{n}^{\alpha_{n}}$ for brevity. From \eqref{PBW-basis}, we obtain another $\mathbb{K}$-basis of $\widehat{Y}_{r,n}$:
\begin{equation}\label{PBW-basis2}
\big\{E_{\chi}x^{\alpha}f_{w}\:|\:\chi\in \text{Irr}(\mathcal{T}_{r, n}), \alpha=(\alpha_1,\ldots,\alpha_{n})\in \mathbb{Z}_{\geq 0}^{n}, w\in \mathfrak{S}_{n}\big\};\end{equation}
and from \eqref{CYPBW-basis}, we obtain another $\mathbb{K}$-basis of $Y_{r, n}^{d}$:
\begin{equation}\label{PBW-basis2}
\big\{E_{\chi}x^{\beta}f_{w}\:|\:\chi\in \text{Irr}(\mathcal{T}_{r, n}), \beta=(\beta_1,\ldots,\beta_{n})\text{ with each }0\leq \beta_i\leq d-1, w\in \mathfrak{S}_{n}\big\}.\end{equation}

An $r$-tuple $\mu=(\mu_1,\ldots,\mu_{r})\in \mathbb{Z}_{\geq 0}^{r}$ such that $\Sigma_{1\leq a\leq r}\mu_{a}=n$ is called an $r$-composition of $n,$ which is denoted by $\mu \models n$. Let $\mathcal{C}_{r,n}$ be the set of $r$-compositions of $n.$ Assume that $\chi\in \text{Irr}(\mathcal{T}_{r, n}).$ For $a\in \{1,\ldots,r\},$ let $\mu_a$ be the number of elements $j\in \{1,\ldots,n\}$ such that $\chi(t_j)=\zeta_a.$ Then the sequence $(\mu_1,\ldots,\mu_r)\in \mathcal{C}_{r,n}$, and we denote it by $\text{Comp}(\chi).$

For each $\mu \models n,$ we define a particular character $\chi_{0}^{\mu}\in \text{Irr}(\mathcal{T}_{r, n})$ by
\begin{equation}\label{chi0-mu}
\left\{\begin{array}{ccccccc}
\chi_0^{\mu} (t_1)&=&\ldots & =& \chi_0^{\mu} (t_{\mu_1})&=& \zeta_1\ ,\\[0.2em]
\chi_0^{\mu} (t_{\mu_1+1})&=&\ldots & =& \chi_0^{\mu} (t_{\mu_1+\mu_2})&=& \zeta_2\ ,\\
\vdots &\vdots &\vdots &\vdots &\vdots &\vdots &\vdots  \\
\chi_0^{\mu} (t_{\mu_1+\dots+\mu_{r-1}+1})&=&\ldots & =& \chi_0^{\mu} (t_{n})&=& \zeta_r\ .\\
\end{array}\right.
\end{equation}
Notice that $\text{Comp}(\chi_{0}^{\mu})=\mu.$ From \eqref{chi0-mu}, we see that the stabilizer of $\chi_{0}^{\mu}$ under the action of $\mathfrak{S}_{n}$ is the Young subgroup $\mathfrak{S}^{\mu},$ which is defined to be $\mathfrak{S}_{\mu_1}\times\cdots\times\mathfrak{S}_{\mu_r}$. Notice that there is a unique representative of minimal length in each left coset in $\mathfrak{S}_{n}/\mathfrak{S}^{\mu}.$ Thus, for any $\chi\in \text{Irr}(\mathcal{T}_{r, n})$ such that $\text{Comp}(\chi)=\mu$, we define a permutation $\pi_{\chi}\in \mathfrak{S}_{n}$ by requiring that $\pi_{\chi}$ is the distinguished left coset representative such that
\begin{equation}\label{chi-mu}
\pi_{\chi}(\chi_{0}^{\mu})=\chi.
\end{equation}

\subsection{Isomorphism theorem for degenerate affine and cyclotomic Yokonuma-Hecke algebras}

For each $\mu \models n,$ we denote by $\widehat{H}^{\mu}$ the algebra $\widehat{H}_{\mu_1}\otimes\cdots\otimes \widehat{H}_{\mu_r},$ which is a free $\mathbb{K}$-module with basis $\{\bar{x}^{\alpha}\bar{w}\:|\:\alpha\in \mathbb{Z}_{\geq 0}^{n}, w\in \mathfrak{S}^{\mu}\}$, and we set $m_{\mu}$ to be the index of the Young subgroup $\mathfrak{S}^{\mu}$ in $\mathfrak{S}_{n}$, that is, $m_{\mu}=m!/(\mu_{1}!\cdots\mu_{r}!).$

For each $\mu \models n,$ let $\text{Mat}_{m_{\mu}}(\widehat{H}^{\mu})$ be the algebra of matrices of size $m_{\mu}$ with coefficients in $\widehat{H}^{\mu}.$ It is easy to see that $m_{\mu}$ is also the number of characters $\chi\in \text{Irr}(\mathcal{T}_{r, n})$ such that $\text{Comp}(\chi)=\mu$. Thus, we can label the rows and columns of a matrix in $\text{Mat}_{m_{\mu}}(\widehat{H}^{\mu})$ by such characters. Moreover, for two characters $\chi, \chi'$ such that $\text{Comp}(\chi)=\text{Comp}(\chi')=\mu$, we denote by $\bm{1}_{\chi, \chi'}$ the matrix in $\text{Mat}_{m_{\mu}}(\widehat{H}^{\mu})$ with $1$ in line $\chi$ and column $\chi'$, and $0$ elsewhere.

For each $\mu \models n,$ we set
\begin{equation*}
E_{\mu} :=\sum_{\text{Comp}(\chi)=\mu}E_{\chi}.
\end{equation*}
Then the set $\{E_{\mu}\:|\:\mu\in \mathcal{C}_{r,n}\}$ forms a complete set of pairwise orthogonal central idempotents in $\widehat{Y}_{r,n}$. In particular, we have the following decomposition of $\widehat{Y}_{r,n}$ into a direct sum of two-sided ideals:
\begin{equation}\label{direct-sum}
\widehat{Y}_{r,n}=\bigoplus_{\text{Comp}(\chi)=\mu}E_{\mu}\widehat{Y}_{r,n}.
\end{equation}

Let $\mu \models n.$ We define a linear map
\begin{equation*}
\Phi_{\mu}: E_{\mu}\widehat{Y}_{r,n}\rightarrow \text{Mat}_{m_{\mu}}(\widehat{H}^{\mu})
\end{equation*}
by
\begin{equation}\label{Phi-mu}
\Phi_{\mu}(E_{\chi}x^{\alpha}f_{w})=\bm{1}_{\chi, w^{-1}(\chi)}\bar{x}^{\pi_{\chi}^{-1}(\alpha)}\overline{\pi_{\chi}^{-1}w\pi_{w^{-1}(\chi)}}
\end{equation}
for $\chi\in \text{Irr}(\mathcal{T}_{r, n})$ such that $\text{Comp}(\chi)=\mu$, $\alpha\in \mathbb{Z}_{\geq 0}^{n}$ and $w\in \mathfrak{S}_{n}.$

We also define a linear map
\begin{equation*}
\Psi_{\mu}: \text{Mat}_{m_{\mu}}(\widehat{H}^{\mu})\rightarrow E_{\mu}\widehat{Y}_{r,n}
\end{equation*}
by
\begin{equation}\label{Psi-mu}
\Psi_{\mu}(\bm{1}_{\chi, \chi'}\bar{x}^{\alpha}\bar{w})=E_{\chi}x^{\pi_{\chi}(\alpha)}f_{\pi_{\chi}w\pi_{\chi'}^{-1}}
\end{equation}
for $\chi, \chi'\in \text{Irr}(\mathcal{T}_{r, n})$ such that $\text{Comp}(\chi)=\text{Comp}(\chi')=\mu$, $\alpha\in \mathbb{Z}_{\geq 0}^{n}$ and $w\in \mathfrak{S}^{\mu}.$

We define the linear maps $\Phi_{r, n} :=\bigoplus_{\mu\in \mathcal{C}_{r,n}}\Phi_{\mu}$ and $\Psi_{r, n} :=\bigoplus_{\mu\in \mathcal{C}_{r,n}}\Psi_{\mu}.$ The following theorem can be proved in exactly the same way as in [PA, Theorem 3.1] (much easier).
\begin{theorem}\label{isomorphsim-theorem1}
Let $\mu \models n.$ The linear map $\Phi_{\mu}$ is an isomorphism of algebras with the inverse map $\Psi_{\mu}.$ Accordingly, $\Phi_{r, n}$ defines an isomorphism between the degenerate affine Yokonuma-Hecke algebra $\widehat{Y}_{r,n}$ and $\bigoplus_{\mu\in \mathcal{C}_{r,n}}\emph{Mat}_{m_{\mu}}(\widehat{H}^{\mu})$ with the inverse map $\Psi_{r, n}.$
\end{theorem}
\begin{proof}
It is easy to check that $\Phi_{\mu}\circ \Psi_{\mu}=\text{Id}$ and $\Psi_{\mu}\circ \Phi_{\mu}=\text{Id}$. To check that $\Psi_{\mu}$ is an algebra homomorphism, we need the following crucial equality:
\begin{equation}\label{EXG}
E_{\chi}x^{\pi_{\chi}(\alpha)}f_{\pi_{\chi}}=E_{\chi}f_{\pi_{\chi}}x^{\alpha}\qquad \text{for any } \text{Comp}(\chi)=\mu
\end{equation}
and the next lemma. We omit the details.
\end{proof}
\begin{lemma}
Let $\mu \models n.$ There exists an algebra isomorphism \[\phi_{\mu}: \widehat{H}^{\mu}\overset{\sim}{\longrightarrow} E_{\chi_{0}^{\mu}}\widehat{Y}_{r,n}E_{\chi_{0}^{\mu}}, \] which is defined by $\phi_{\mu}(\bar{x}^{\alpha}\bar{w})=E_{\chi_{0}^{\mu}}{x}^{\alpha}f_{w}$ for $\alpha\in \mathbb{Z}_{\geq 0}^{n}$ and $w\in \mathfrak{S}^{\mu}.$
\end{lemma}
\begin{remark}\label{remark}
The equality \eqref{EXG} can be regarded as a degenerate case of the next one
\begin{equation}\label{EXG1}
E_{\chi}X^{\pi_{\chi}(\alpha)}g_{\pi_{\chi}}=E_{\chi}g_{\pi_{\chi}}X^{\alpha}\qquad \text{for any } \text{Comp}(\chi)=\mu,
\end{equation}
which plays an important role in the proof of [PA, Theorem 3.1].
\end{remark}

For $\mu \models n,$ we set $H^{\mu} :=H_{\mu_{1}}^{d}\otimes\cdots\otimes H_{\mu_{r}}^{d}$. By definition, $H^{\mu}$ is the quotient of the algebra $\widehat{H}^{\mu}$ by the two-sided ideal generated by the elements
\begin{equation}\label{elements}
(\bar{x}_{\mu_1+\cdots+\mu_{a-1}+1}-v_1)\cdots (\bar{x}_{\mu_1+\cdots+\mu_{a-1}+1}-v_{d}),\qquad a=1,\ldots,r.
\end{equation}
The following theorem can be proved in exactly the same way as in [PA, Corollary 3.2].
\begin{theorem}\label{isomorphsim-theorem2}
There exists an algebra isomorphism between the degenerate cyclotomic Yokonuma-Hecke algebra $Y_{r, n}^{d}$ and the direct sum $\bigoplus_{\mu\in \mathcal{C}_{r,n}}\emph{Mat}_{m_{\mu}}(H^{\mu}).$
\end{theorem}

\begin{remark}\label{remark2}
Rostam [Ro, Theorem 5.15] has established an isomorphism between the degenerate cyclotomic Yokonuma-Hecke algebra and some cyclotomic quiver Hecke algebra associated to a quiver given by disjoint copies of cyclic quivers. In [Ro, Theorem 6.30], he has proved a general isomorphism theorem on cyclotomic quiver Hecke algebras, where the associated quiver is provided by a disjoint union of full subquivers. Combined with the isomorphism between the degenerate cyclotomic Hecke algebra and the cyclotomic quiver Hecke algebra associated to a cyclic quiver established by Brundan and Kleshchev [BK2, Theorem 1.1], it concludes that we also obtain an isomorphism between the degenerate cyclotomic Yokonuma-Hecke algebra and a direct sum of matrix algebras over degenerate cyclotomic Hecke algebras.

By an argument analogous to the proof of [Ro, Theorem 6.34], we see that the isomorphism obtained above coincides with the one in Theorem \ref{isomorphsim-theorem2}. That is, in our situation, the analogues of [Ro, Theorems 6.34 and 6.35] also hold.
\end{remark}

\section{Applications}

In this section we present several applications of the isomorphism theorems established in the preceding section.

\subsection{Simple modules}

By [La, Theorem 17.20], an algebra $S$ and a matrix algebra $R=\text{M}_{n}(S)$ over $S$ (for any fixed $n\geq 1$) are Morita equivalent. For a finite dimensional linear space $V,$ we denote by $V^{(n)}$ the space of $n$-tuples $(v_1,\ldots,v_n)$ with each $v_i\in V$ $(1\leq i\leq n).$

Thus, by Theorem \ref{isomorphsim-theorem1}, any simple $\widehat{Y}_{r,n}$-module is of the form \[(M_1\otimes\cdots\otimes M_r)^{(m_{\mu})},\]
where $\mu=(\mu_1,\ldots,\mu_r)\in \mathcal{C}_{r,n}$ and each $M_k$ is a simple $\widehat{H}_{\mu_k}$-module for $1\leq k\leq r.$ Therefore, we recover the modular representation theory of $\widehat{Y}_{r,n}$ established in [WW, Theorem 4.4].

Similarly, by Theorem \ref{isomorphsim-theorem2}, any simple $Y_{r, n}^{d}$-module is of the form \[(\overline{M}_1\otimes\cdots\otimes \overline{M}_r)^{(m_{\mu})},\]
where $\mu=(\mu_1,\ldots,\mu_r)\in \mathcal{C}_{r,n}$ and each $\overline{M}_k$ is a simple $H_{\mu_{k}}^{d}$-module for $1\leq k\leq r.$ Therefore, we recover the modular representation theory of $Y_{r, n}^{d}$ established in [WW, Theorem 5.12].

By Theorem 3.4, $Y_{r, n}^{d}$ is split semisimple if and only if for all $\mu\in \mathcal{C}_{r,n}$, the algebra $H^{\mu}$ is split semisimple. By [AMR, Theorem 6.11], this happens if and only if
\begin{equation}\label{semisimple-condition}
n!\prod_{1\leq i< j\leq d}\prod_{-n< l< n}(l+v_i-v_j)\neq 0.
\end{equation}

Since the tensor product of two cellular algebras is still a cellular algebra and the degenerate cyclotomic Hecke algebra $H_{n}^{d}$ is cellular by [AMR, Theorem 6.3], we get that $Y_{r, n}^{d}$ is a cellular algebra by Theorem \ref{isomorphsim-theorem2}.

\subsection{Schur elements}
In this subsection, we assume that $Y_{r, n}^{d}$ is split semisimple, that is, \eqref{semisimple-condition} is satisfied. We identify $H^{\mu}$ with $H_{\mu_{1}}^{d}\otimes\cdots\otimes H_{\mu_{r}}^{d}$. Recall the symmetrizing form defined in \eqref{symme-form}, which we can use to define a symmetrizing form $\tau^{\mu}$ on $H^{\mu}$ by
\begin{align}\label{symme-from1}
&\tau^{\mu}(\bar{x}_{1}^{i_1}\cdots \bar{x}_{\mu_1}^{i_{\mu_1}}\bar{w_1}\otimes\cdots\otimes \bar{x}_{\mu_1+\cdots+\mu_{r-1}+1}^{i_{\mu_1+\cdots+\mu_{r-1}+1}}\cdots \bar{x}_{n}^{i_{n}}\bar{w_r})\notag\\
=&\tau_{\mu_1}(\bar{x}_{1}^{i_1}\cdots \bar{x}_{\mu_1}^{i_{\mu_1}}\bar{w_1})\cdots \tau_{\mu_r}(\bar{x}_{\mu_1+\cdots+\mu_{r-1}+1}^{i_{\mu_1+\cdots+\mu_{r-1}+1}}\cdots \bar{x}_{n}^{i_{n}}\bar{w_r}).
\end{align}

For any $n\geq1,$ $\lambda=(\lambda_{1},\ldots,\lambda_{k})$ is called a partition of $n$ if it is a finite sequence of non-increasing nonnegative integers whose sum is $n.$ We write $\lambda\vdash n$ if $\lambda$ is a partition of $n,$ and we define $|\lambda| :=n$. A $d$-partition of $n$ is an ordered $d$-tuple $\bm{\lambda}=(\lambda^{(1)},\lambda^{(2)},\ldots,\lambda^{(d)})$ of partitions $\lambda^{(k)}$ such that $|\bm{\lambda}| :=\sum_{k=1}^{d}|\lambda^{(k)}|=n.$

It is known that the simple modules of the split semisimple algebra $H_{n}^{d}$ are labelled by $d$-partitions of $n.$ For $\bm{\lambda}$ a $d$-partition of $n$, let $\overline{M}_{\bm{\lambda}}$ be the corresponding simple module of $H_{n}^{d}.$ We denote by $s_{\bm{\lambda}} :=s_{\overline{M}_{\bm{\lambda}}}$ the Schur element of $\overline{M}_{\bm{\lambda}}$ associated to $\tau_{n},$ which has been explicitly calculated in [Z1, Theorem 4.2 and Z2, Theorem 5.5].

Let $\mu\in \mathcal{C}_{r,n}$ and let $\underline{\bm{\lambda}}=(\bm{\lambda}^{1},\ldots,\bm{\lambda}^{r})$ be an $r$-tuple of $d$-partitions such that $\mu=(|\bm{\lambda}^{1}|,\ldots,|\bm{\lambda}^{r}|)$. Set $\overline{M}_{\underline{\bm{\lambda}}} :=\overline{M}_{\bm{\lambda}^{1}}\otimes\cdots\otimes\overline{M}_{\bm{\lambda}^{r}}.$ We see that $\overline{M}_{\underline{\bm{\lambda}}}$ is a simple $H^{\mu}$-module. Moreover, by \eqref{symme-from1}, the Schur element $s_{\underline{\bm{\lambda}}}$ of $\overline{M}_{\underline{\bm{\lambda}}}$ associated to $\tau^{\mu}$ is given by
\begin{align}\label{symme-from2}
s_{\underline{\bm{\lambda}}}=s_{\bm{\lambda}^{1}}\cdots s_{\bm{\lambda}^{r}}.
\end{align}

Now let us consider the algebra $\bigoplus_{\mu\in \mathcal{C}_{r,n}}\text{Mat}_{m_{\mu}}(H^{\mu}).$ By [JaPA, Lemma 4.4(i)], we obtain a symmetrizing form on it, which is given by $\bigoplus_{\mu\in \mathcal{C}_{r,n}}\tau^{\mu}\circ \text{Tr}_{\text{Mat}_{m_{\mu}}}$. Moreover, by [JaPA, Lemma 4.4(ii)], the associated Schur element of the simple module, which is indexed by an $r$-tuple of $d$-partitions of $n,$ is given by the formula \eqref{symme-from2}.

Finally, let us consider the split semisimple degenerate cyclotomic Yokonuma-Hecke algebra $Y_{r, n}^{d}$. By the discussion in Subsection 4.1, we see that the simple $Y_{r, n}^{d}$-modules are indexed by the set of $r$-tuples of $d$-partitions of $n$. Assume that $\underline{\bm{\lambda}}=(\bm{\lambda}^{1},\ldots,\bm{\lambda}^{r})$ is an $r$-tuple of $d$-partitions of $n.$ By Theorem \ref{isomorphsim-theorem2} and the preceding discussion, we obtain naturally a symmetrizing form on $Y_{r, n}^{d},$ which is given by
\begin{align}\label{symme-from3}
\rho_{n} :=\bigoplus_{\mu\in \mathcal{C}_{r,n}}\tau^{\mu}\circ \text{Tr}_{\text{Mat}_{m_{\mu}}}\circ \bar{\Phi}_{\mu},
\end{align}
where $\bar{\Phi}_{\mu}$ is the induced linear map on $E_{\mu}Y_{r, n}^{d}$ by $\Phi_{\mu}$ defined in \eqref{Phi-mu}. And moreover, the associated Schur element of the simple $Y_{r, n}^{d}$-module, which is indexed by $\underline{\bm{\lambda}},$ is given by the formula \eqref{symme-from2}.

\subsection{Alternative formula for $\rho_{n}$}

We consider the following linear form $\hat{\rho}_{n}:Y_{r, n}^{d}\rightarrow \mathbb{K}$ on $Y_{r, n}^{d}$:
\begin{align}\label{symme-form4}
&\hat{\rho}_{n}(t_{1}^{s_1}\cdots t_{n}^{s_{n}}x_{1}^{r_1}\cdots x_{n}^{r_{n}}f_{w})\notag\\
=&
\begin{cases}
r^{n} & \text{if } r_{1}=\cdots=r_{n}=d-1, s_1\equiv\cdots\equiv s_{n}\equiv 0~ (\text{mod }r)\text{ and }w=1,
\\
0 & \text{otherwise. }
\end{cases}
\end{align}
It turns out that $\hat{\rho}_{n}$ actually coincides with the natural symmetrizing from $\rho_{n}.$
\begin{proposition}\label{coincide}
The form $\hat{\rho}_{n}$ coincides with the symmetrizing from $\rho_{n}$ on $Y_{r, n}^{d}.$
\end{proposition}
\begin{proof}
It suffices to show that the two forms take the same value on the basis given in \eqref{PBW-basis2}. Fix $\mu\in \mathcal{C}_{r,n}$ and some $\chi$ such that $\text{Comp}(\chi)=\mu$, $\beta=(\beta_1,\ldots,\beta_{n})$ with each $0\leq \beta_i\leq d-1$ and $w\in \mathfrak{S}_{n}$. We use \eqref{idempotents} to get that
\begin{align*}
\hat{\rho}_{n}(E_{\chi}x^{\beta}f_{w})&=\hat{\rho}_{n}\bigg(\Big(\prod_{1\leq i\leq n}\frac{1}{r}\sum_{0\leq s\leq r-1}\chi(t_{i})^{s}t_{i}^{-s}\Big)x_{1}^{\beta_1}\cdots x_{n}^{\beta_{n}}f_{w}\bigg)\\
&=\hat{\rho}_{n}\bigg(\Big(\prod_{1\leq i\leq n}\frac{1}{r}\Big)x_{1}^{\beta_1}\cdots x_{n}^{\beta_{n}}f_{w}\bigg)\\
&=
\begin{cases}
1 & \text{if } \beta_{1}=\cdots=\beta_{n}=d-1\text{ and }w=1,
\\
0 & \text{otherwise. }
\end{cases}
\end{align*}

On the other hand, by \eqref{symme-from3} we have
\begin{align*}
\rho_{n}(E_{\chi}x^{\beta}f_{w})&=\tau^{\mu}\circ \text{Tr}_{\text{Mat}_{m_{\mu}}}\circ \bar{\Phi}_{\mu}(E_{\chi}x^{\beta}f_{w})\\
&=\tau^{\mu}\circ \text{Tr}_{\text{Mat}_{m_{\mu}}}\big(\bm{1}_{\chi, w^{-1}(\chi)}\bar{x}^{\pi_{\chi}^{-1}(\beta)}\overline{\pi_{\chi}^{-1}w\pi_{w^{-1}(\chi)}}\big).
\end{align*}
We have $w^{-1}(\chi)=\chi$ if and only if $\pi_{\chi}^{-1}w\pi_{\chi}\in \mathfrak{S}^{\mu}$. By the equality above, \eqref{symme-from1} and the fact that $\pi_{\chi}^{-1}w\pi_{\chi}=1$ if and only if $w=1,$ we have
\begin{align*}
\rho_{n}(E_{\chi}x^{\beta}f_{w})&=\tau^{\mu}\big(\bar{x}^{\pi_{\chi}^{-1}(\beta)}\overline{\pi_{\chi}^{-1}w\pi_{\chi}}\big)\\
&=\begin{cases}
1 & \text{if } \beta_{1}=\cdots=\beta_{n}=d-1\text{ and }w=1,
\\
0 & \text{otherwise. }
\end{cases}
\end{align*}
We have proved this proposition.
\end{proof}
\begin{remark}\label{remark}
Proposition \ref{coincide} immediately implies that $\hat{\rho}_{n}$ is a non-degenerate trace on $Y_{r, n}^{d}.$ Thus, we have given another proof of Theorem \ref{symmetric-algebra}.
\end{remark}

\noindent{\bf Acknowledgements.}
The author is deeply indebted to Dr. S. Rostam for pointing out the fact that Theorem \ref{isomorphsim-theorem2} can be deduced from the results in the latest version of [Ro]. Many ideas of this paper originate from the reference [PA].



\end{document}